\documentclass[twocolumn,fleqn,natbib]{svjour2}
\smartqed  % flush right qed marks, e.g. at end of proof
\usepackage{graphicx}
\usepackage{mathptmx}      % use Times fonts if available on your TeX system
\usepackage{amssymb,amsmath}
\usepackage{tikz}
\usetikzlibrary{arrows}
\newcommand{\RR}{\mathbb{R}}
\newtheorem{algorithm}[theorem]{Algorithm}

\journalname{Structural and Multidisciplinary Optimization}
\begin{document}
\title{On Robustness Criteria and Robust Topology Optimization with Uncertain
Loads\thanks{This research was supported
by the EU FP7 project AMAZE.} }
%%%%\subtitle{Do you have a subtitle?\\ If so, write it here}
%
\titlerunning{Robust Topology Optimization with Uncertain Loads}  % if too long for running head
\author{Michal Ko\v{c}vara}
\institute{ M. Ko\v{c}vara \at  School of Mathematics, The University
of Birmingham, Birmingham B15 2TT, Great Britain, and Institute of
Information Theory and Automation, Academy of Sciences of the Czech
Republic, Pod
vod\'arenskou v\v{e}\v{z}\'{\i}~4, 18208 Prague 8\\
\email{m.kocvara@bham.ac.uk} }
\date{Received: date / Revised: date}
\maketitle
\begin{abstract}
We propose a new algorithm for the solution of the robust multiple-load
topology optimization problem. The algorithm can be applied to any type
of problem, e.g., truss topology, variable thickness sheet or free
material optimization. We assume that the given loads are uncertain and
can be subject to small random perturbations. Furthermore, we define a
rigorous measure of robustness of the given design with respect to
these perturbations. To implement the algorithm, the users only need
software to solve their standard multiple-load problem. Additionally,
they have to solve a few small-dimensional eigenvalue problems.
Numerical examples demonstrate the efficiency of our approach.
\end{abstract}
\bibliographystyle{spbasic}
\section{Introduction}\label{sec:1}
This article has been motivated by the following sentence of an
engineer in an industrial company: ``When we use off-the-shelf topology
optimization software, we always consider not only the nominal loads
but also their angular perturbations by up to 30 degrees.'' The goal of
this article is to automatize this heuristics and to give rigorous
measures of robustness of a structure with respect to these
perturbations.

Robust topology optimization (in fact, any robust optimization problem)
can be approached from two different angles---a stochastic one and a
deterministic one. Most of the existing literature deal with the
stochastic approach
\citep[e.g.][]{evgrafov2003stochastic,doltsinis2004robust,conti2009shape}.
The deterministic (worst case) approach has been pioneered by Ben-Tal,
Nemirovksi and El Ghaoui
\citep{ben1997robust,ben-tal-nemirovski,el1997robust,ben2009robust}. In
their monograph, \cite{ben-tal-nemirovski} defined the concept of a
robust counterpart to a nominal (convex) optimization problem, where
the problem data is assumed to live in an uncertainty set.
\cite{ben-tal-nemirovski} showed that if the uncertainty set is an
ellipsoid, then the robust counterpart (a semi-infinite optimization
problem) can be formulated as a computationally tractable convex cone
optimization problem. In the same monograph, they presented explicit
formulations of robust counterparts for the truss topology and the free
material optimization problems with uncertainty in the loadings.
Unfortunately, these problems (typically large-scale linear
semidefinite optimization problems) are just too large to be
computationally tractable in practical situations. For this reason, in
\cite*{kocvara-zowe-nemirovski} we have developed a so-called cascading
technique that reduces the dimension of the robust counterpart
significantly. This article makes an attempt to go one step further in
bringing the solution of the robust topology optimization problem
closer to use in engineering practice.

After introducing the notation and the standard multiple-load topology
optimization problem in Section~\ref{sec:2}, we describe the main idea
of our approach and the corresponding algorithm in Section~\ref{sec:3}.
Section~\ref{sec:ex} is devoted to numerical experiments.

In the article we use standard notation for vectors and matrices: $x_i$
is the $i$-th element of vector $x\in\RR^n$ and $A_{ij}$ an $(i,j)$-th
element of matrix $A\in\RR^{n\times m}$. If $I\subset\{1,2,\ldots,n\}$,
$J\subset\{1,2,\ldots,n\}$ are sets of indices, then $x_I$ is a
subvector of $x$ with indices from $I$ and $A_{IJ}$ a submatrix of $A$
with row indices from $I$ and column indices from $J$. For $x\in\RR^n$,
$\|x\|$ denotes the Euclidean norm of $x$.

\section{Problem definition}\label{sec:2}
%\subsection{Basic notations}\label{sec:2.1}
%
We consider a general mechanical structure, discrete or discretized by
the finite element method. The number of members or finite elements is
denoted by $m$, the total number of ``free'' degrees of freedom (i.e.,
not fixed by Dirichlet boundary conditions) by $n$.

For a given set of $L$ (independent) load vectors
\begin{equation}\label{eq:fneq0}
f^{(\ell)}\in\RR^n,\;\;f^{(\ell)}\neq 0,\qquad \ell=1,\ldots,L,
\end{equation}
the structure should satisfy linear equilibrium equations
\begin{equation}\label{eq:eleq0}
 K(x) u^{(\ell)} = f^{(\ell)}, \qquad \ell=1,\ldots,L.
\end{equation}
Here $K(x)$ is the stiffness matrix of the structure, depending on a
design variable $x$.

We do not assume any particular structure of $K(x)$ or its dependence
on $x$. Therefore, the problem formulations and the conclusions apply
to a broad class of problems, e.g., the truss topology optimization,
variable thickness sheet, SIMP and free material optimization (see,
e.g., \citet{bendsoe}). All we need is software for the solution of the
specific multiple-load problem. Consequently, the design variables
$x\in\RR^m$, $x\geq 0$, represent, for instance, the thickness,
cross-sectional area or material properties of the element.

Let
\[
X:=\{x\in\RR^m\mid \sum\limits_{i=1}^mx_i  \leq v;\
\underline{x}\leq x_i \leq \overline{x},\ i=1,\ldots,m\}
\]
be the set of feasible design variables with some
$v,\underline{x},\overline{x}\in\RR$, $v>0$ and $0\leq\underline{x}\leq
\overline{x}$ (again, the specific form of this set is not important
for our purposes). The standard formulation of the worst-case
multiple-load topology optimization problem reads as follows:
\begin{alignat}{2}\label{eq:minc}
     &\min_{x\in X,u\in\RR^{L\cdot n}} \max_{\ \ell=1,\ldots,L\ } (f^{(\ell)})^Tu^{(\ell)} \\
     &\mbox{subject to} \notag\\
     &\qquad K(x) u^{(\ell)} = f^{(\ell)},\quad\ell=1,\ldots,L\,.\notag
\end{alignat}
%
%\begin{alignat}{2}\label{eq:minc1}
%     &\min_{x\in\RR^m,u\in\RR^{L\cdot n},\gamma\in\RR} \gamma \\
%     &\mbox{subject to} \notag\\
%     &\qquad K(x) u^{(\ell)} = f^{(\ell)},&\quad\ell&=1,\ldots,L\notag\\
%     &\qquad f_\ell^T u_\ell  \leq \gamma,                   &\quad\ell&=1,\ldots,L\notag\\
%     &\qquad \sum\limits_{i=1}^mx_i  \leq v\notag\\
%     &\qquad 0\leq \underline{x}\leq x_i \leq \overline{x},  &\quad   i&=1,\ldots,m \,.\notag
%\end{alignat}
%
To simplify our notation, we will instead consider the following
``nested'' formulation
\begin{alignat}{2}\label{eq:ml}
     &\min_{x\in X} \max_{\ \ell=1,\ldots,L\ }
     (f^{(\ell)})^T K(x)^{-1}f^{(\ell)}\,,
\end{alignat}
where, in case of $K(x)$ singular, we consider the generalized
Moore-Penrose inverse of the matrix. Note that, for the numerical
treatment, one would use the equivalent formulation
\begin{alignat}{2}\label{eq:ml1}
     &\min_{x\in X,\gamma\in\RR} \gamma \\
     &\mbox{subject to} \notag\\
     &\qquad (f^{(\ell)})^T K(x)^{-1}f^{(\ell)}  \leq \gamma, \quad\ell=1,\ldots,L\,.\notag
\end{alignat}
In the following, we will use formulation (\ref{eq:ml}). This is just
for the sake of keeping the notation fixed. In practical
implementation, the users can use any multiple-load formulation
implemented in their software.

\section{Robust topology optimization}\label{sec:3}
\subsection{General approach}
In their ground-breaking theory of robust convex optimization
\cite{ben-tal-nemirovski} define a \emph{robust counterpart} to a
nominal convex optimization problem in the worst-case sense. The
solution of the robust problem should be feasible for \emph{any}
instance of the random data and the optimum is attained at the maximum
of the objective function over all these instance.
\cite{ben-tal-nemirovski} show that if the data of the problem
(vectors, matrices) lie in \emph{ellipsoidal uncertainty sets}, the
robust counterpart---essentially a semi-infinite optimization
problem---can be converted into a numerically tractable (solvable in
polynomial time) convex optimization problem.

Specifically, if we assume uncertainty in the loads of our topology
optimization problem (\ref{eq:ml}), the robust counterpart is defined
as
\begin{equation}\label{eq:robbtn}
  \min_{x\in X} \max_{\ \ell=1,\ldots,L\ } \max_{f\in {\cal U}_\ell}
  f^T K(x)^{-1}f\,.
\end{equation}
where
\begin{equation}\label{eq:robbtn_u}
  {\cal U}_\ell :=\left\{ f\mid \exists g\in\RR^p,\
  \|g\|\leq 1: f=f_0^{(\ell)} + \sum_{i=1}^p g_i f_i^{(\ell)}  \right\}\,;
\end{equation}
here $f_0^{(\ell)}$ are the nominal loads and $f_i^{(\ell)}\in\RR^n,\
i=1,\ldots,p$, define an ellipsoid around $f_0^{(\ell)}$.
\cite{ben-tal-nemirovski} have shown that (\ref{eq:robbtn}) with the
uncertainty set (\ref{eq:robbtn_u}) can be formulated as a linear
semidefinite optimization problem. Unfortunately, in the context of
topology optimization, the dimension of this problem may be very large:
basically, it is the number of the finite element nodes times the space
dimension.

To avoid the problem of the prohibitive dimension, in
\cite{kocvara-zowe-nemirovski} we have proposed a \emph{cascading
algorithm} that leads to an approximate solution of the original robust
problem. The idea is to find only the ``most dangerous'' incidental
loads and to solve the robust problem only with these dangerous loads,
ignoring the others. In this article, we took inspiration from
\cite{kocvara-zowe-nemirovski}; however, we have substantially modified
the uncertainty sets which also leads to a modification of the
algorithm. Our goal was to get closer to engineering practice and to
make the approach usable for practitioners.

\subsection{Uncertainty set}
In \citet{kocvara-zowe-nemirovski} we have considered random
perturbations of loads at any free node of the finite element mesh (or
truss). This leads not only to very large dimensional robust
counterparts but also to practical difficulties when a perturbation
force can be applied to a node that would not normally be a part of the
optimal structure.

In this article we are motivated by the practice when, instead of
considering only the nominal loads, the engineers also apply these very
loads but in slightly perturbed directions. Our goal is to automatize
this heuristics and to give rigorous measures of robustness of a
particular design with respect to these perturbations.

Consider the multiple-load topology optimization problem (\ref{eq:ml})
with loads $f^{(\ell)}$, $\ell=1,\ldots,L$. We assume that the loads
are applied at certain nodes, either the nodes of the truss ground
structure or nodes of the finite-element discretization. Each node
$\nu_i$, $i=1,\ldots,N$, is associated with $d$ degrees of freedom
$\nu_{i_1},\ldots,\nu_{i_d}$. Typically, $d$ is equal to the spatial
dimension of the problem. As we have $n$ degrees of freedom, we assume
that they can be order such that
\[
  \{\nu_{1_1},\ldots,\nu_{1_d},\nu_{2_1},\ldots,\nu_{2_d},\ldots\ldots,
  \nu_{N_1},\ldots,\nu_{N_d}\} =\{1,\ldots,n\}\,.
\]

For each $f^{(\ell)}$ we find the set of indices of nodes with at least
one non-zero component of $f_0^{(\ell)}$
\[
  \hat{I}_\ell:=\{i\mid (f_0^{(\ell)})_{\nu_{i_j}}\not= 0\quad\mbox{for some}\ j=1,\ldots,d\}\,,
\]
the set of the corresponding degrees of freedom
\begin{equation}\label{eq:I}
  I_\ell:=\{k\mid k=\nu_{i_j},\ i\in \hat{I}_\ell,\ j=1,\ldots,d\}
\end{equation}
and its complement in $\{1,\ldots,n\}$:
\begin{equation}\label{eq:J}
  J_\ell:=\{1,\ldots,n\}\setminus I_\ell\,.
\end{equation}

Assume that instead of knowing each of the loads $f^{(\ell)}$ exactly,
we only know that they lie in an ellipsoid around some \emph{nominal
loads} $f_0^{(\ell)}$, $\ell=1,\ldots,L$:
\begin{equation}\label{eq:unc}
  f^{(\ell)} = f_0^{(\ell)}+P_\ell g,\quad \|g\|\leq 1,\quad g_i=0 \mbox{~if~}i\in J_\ell
\end{equation}
where $P_\ell$ is a symmetric and positive semidefinite matrix with
$(P_\ell)_{ij}=0$ if either $i\in J_\ell$ or $j\in J_\ell$. The choice
of $P_\ell$ is discussed below.

%Typically, we would choose $\tau_\ell = 0.1\|f^{(\ell)}\|$ or
%$\tau_\ell = 0.3\|f^{(\ell)}\|$.

\paragraph{Choice of $P_\ell$}
Consider a nominal load $f_0^{(\ell)}$. Notice first the second part of
the definition (\ref{eq:unc}) concerning the zero components of the
perturbation vector $g$. This means that the perturbed load
$f^{(\ell)}$ is only applied at the same nodes as the nominal load
$f_0^{(\ell)}$. Matrix $P_\ell$ defines the neighbourhood of
$f_0^{(\ell)}$ in which we can expect the random perturbations. Denote
by $\tilde{P}_\ell$ the restriction $(P_\ell)_{I_\ell I_\ell}$. The
choice
\[
  \tilde{P}_\ell = \tau I
\]
defines a ball of radius $\tau$ around $f_0^{(\ell)}$, see
Fig.~\ref{fig:unc}-left. If we want to consider significant angular
perturbation of $f_0^{(\ell)}$ but just a small perturbation in its
magnitude, we would chose $P_\ell$ to define a flat ellipsoid. For
instance, if
\[
  \tilde{P}_\ell = \begin{bmatrix}1.0\cdot 10^{-3}&0\\0&1\end{bmatrix}\quad
  \mbox{for~} f_0^{(\ell)}= (10,\ 0)^T
\]
or, generally,
\[
  \tilde{P}_\ell = T^T\begin{bmatrix}1.0\cdot 10^{-3}d&0\\0&d\end{bmatrix}T\quad
  \mbox{for~} f_0^{(\ell)}= (a,\ b)^T
\]
where $T$ is the rotation matrix for an angle defined by $f_0^{(\ell)}$
\[
  T=\begin{bmatrix}\cos\phi&\sin\phi\\-\sin\phi&\cos\phi\end{bmatrix},\quad\phi=\arctan(b/a)
\]
and $d=\tau\|f_0^{(\ell)}\|$; see Fig.~\ref{fig:unc}-right.
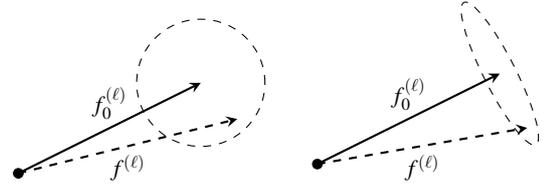
\begin{figure}[h]
\begin{center}
\begin{tikzpicture}
  [scale=1.2,auto=left]
  \draw[fill=black] (0,0) circle (1.5pt);
  \draw[thick,->, >=stealth] (0,0) -- node[above] {$f_0^{(\ell)}$} (2,1) ;
  \draw[dashed] (2,1) circle(20pt);
  \draw[thick,dashed,->, >=stealth ] (0,0) -- node[below] {$f^{(\ell)}$} (2.4,0.6) ;
  %\node at (.9,.8) {$f_0^{(\ell)}$};
  %\node at (1.1,0) {$f^{(\ell)}$};
\end{tikzpicture}\qquad
\begin{tikzpicture}
  [scale=1.2,auto=left]
  \draw[fill=black] (0,0) circle (1.5pt);
  \draw[thick,->, >=stealth] (0,0) -- node[above] {$f_0^{(\ell)}$} (2,1) ;
  \draw[dashed,rotate around={26.6:(2,1)}] (2,1) ellipse (5pt and 25pt);
  \draw[thick,dashed,->, >=stealth ] (0,0) -- node[below] {$f^{(\ell)}$} (2.3,0.4) ;
  %\node at (.9,.8) {$f_0^{(\ell)}$};
  %\node at (1.1,0) {$f^{(\ell)}$};
\end{tikzpicture}
\end{center}
\caption{The nominal load $f_0^{(\ell)}$ and its perturbation $f^{(\ell)}$ for a
circular (left) and ellipsoidal (right) uncertainty set}
\label{fig:unc}
\end{figure}

\subsection{Robust counterpart}
We are now ready to give the definition of the robust counterpart.
\begin{definition}
Consider the multiple-load topology optimization problem (\ref{eq:ml})
with nominal loads $f_0^{(\ell)}$, $\ell=1,\ldots,L$. Define
\begin{equation}
  {\cal G}_\ell :=\{g\in\RR^n\mid \|g\|\leq 1,\ g_i=0 \mbox{~if~}i\in J_\ell\}\,.
\end{equation}
The \emph{robust counterpart} to problem (\ref{eq:ml}) is defined as
\begin{align}\label{eq:rob}
  \min_{x\in X} \max_{\ \ell=1,\ldots,L\ } \max_{g\in{\cal G}_\ell}
  (f_0^{(\ell)}+P_\ell g)^T K(x)^{-1}(f_0^{(\ell)}+P_\ell g)\,.
\end{align}
\end{definition}
So for each load case we consider the worst-case scenario, the ``most
dangerous'' load from the ball around $f_0^{(\ell)}$.

Notice that up to this point we followed the general theory by
\cite{ben-tal-nemirovski}. From now on, we will use the specific form
of the uncertainty set.
In the following, we will show that the most-inner optimization problem
in (\ref{eq:rob}) can be easily solved.

Assume that $x$ and $\ell$ are given. First we find the index sets
$I_\ell$ and $J_\ell$ from (\ref{eq:I}), (\ref{eq:J}). Now we compute
the Schur complement of the inverse stiffness matrix
\begin{equation}\label{eq:schur}
  S^{(\ell)} = K(x)^{-1}_{I_\ell I_\ell} -
  K(x)^{-1}_{J_\ell I_\ell} (K(x)^{-1}_{J_\ell J_\ell})^{-1} K(x)^{-1}_{I_\ell J_\ell}\,.
\end{equation}
We get the obvious statement:
\begin{lemma}\label{th:l1a}
Let $x$ and $\ell$ be given and denote by $\tilde{f} =
(f_0^{(\ell)}){\!_{I_\ell}}\,$. Then
\begin{align}
&\max_{g\in{\cal G}_\ell}
  (f_0^{(\ell)}+P_\ell g)^T K(x)^{-1}(f_0^{(\ell)}+P_\ell g)\label{eq:gmax}\\ &\quad=
\max_{\tilde{g}\in\RR^{|I_\ell|}:\|\tilde{g}\|\leq 1}
  (\tilde{f}+\tilde{P}_\ell \tilde{g})^T S^{(\ell)}(\tilde{f}+\tilde{P}_\ell \tilde{g})\,.\nonumber
\end{align}
\end{lemma}

\begin{lemma}\label{th:l1}
Let $A$ by a symmetric positive semidefinite $n\times n$ matrix and let
$\varphi\in\RR^n$ be given. The optimal value of the problem
\begin{equation}\label{eq:l11}
  \max_{\|\psi\|\leq 1} (\varphi+P\psi)^T A (\varphi+P\psi)
\end{equation}
is attained at the eigenvector $\psi_{\max}$ associated with the
largest eigenvalue $\lambda_{\max}$ of the inhomogeneous eigenvalue
problem
\begin{equation}\label{eq:l12}
  P^TAP\psi + P^TA\varphi = \lambda I \psi\,.
\end{equation}
\end{lemma}
\begin{proof}
The Lagrangian of the constrained optimization problem (\ref{eq:l11})
is given by
\begin{equation*}
  {\cal L}(\psi,\lambda):=(\varphi+P\psi)^T A (\varphi+P\psi) - \lambda(\sum\psi_i^2-1)
\end{equation*}
hence the first order optimality condition reads
\begin{equation*}
  2P^TA(\varphi+P\psi) - 2\lambda I\psi = 0\,.
\end{equation*}
The rest follows from convexity of (\ref{eq:l11}).\qed
\end{proof}
Therefore, by solving the eigenvalue problem
\begin{equation}\label{eq:ev}
  P_\ell^TS^{(\ell)}\tilde{f}+P_\ell^TS^{(\ell)}P_\ell \tilde{g} = \lambda I \tilde{g}
\end{equation}
(with respect to $\tilde{g}$ and $\lambda$) we find the optimal value
of the most-inner problem in (\ref{eq:rob}) and the corresponding
maximizer. Notice that this is a low-dimensional problem, as the number
of non-zeros in $f_0^{(\ell)}$ is typically very small, as compared to
the number of degrees of freedom.

\subsection{Measuring robustness}
Assume that we have solved the original multiple-load problem
(\ref{eq:ml}) with the nominal loads $f_0^{(1)},\ldots,f_0^{(L)}$. Let
us call the optimal design $x^*$. For this design and for each load
case, let us solve the eigenvalue problem (\ref{eq:ev}) to get
eigenvectors $g^{(\ell)}_{\max}$ associated with the largest
eigenvalues $\lambda_{\max}^{(\ell)}$, $\ell=1,\ldots,L$, i.e.,
solutions of (\ref{eq:gmax}). A comparison of the optimal compliance
for the nominal loads with compliances corresponding to these
eigenvectors will give us a clear idea about the vulnerability and
robustness of the design $x^*$.
\begin{definition}\label{def:rob}
Let $x^*$ be the solution of (\ref{eq:ml}) and
\begin{equation*}
c^*:=\max\limits_{\ell=1,\ldots,L}(f_0^{(\ell)})^T
  K(x^*)^{-1}f_0^{(\ell)}
\end{equation*}
the corresponding optimal compliance.
Define
\begin{equation*}
c_{\rm rob}:= \max_{\ell=1,\ldots,L}(f_0^{(\ell)}+P_\ell
g^{(\ell)}_{\max})^T K(x^*)^{-1}(f_0^{(\ell)}+P_\ell g^{(\ell)}_{\max})\,,
\end{equation*}
where $g^{(\ell)}_{\max}$ is a solution of (\ref{eq:gmax}) for
$\ell=1,\ldots,L$. The ratio
\begin{equation*}
  {\cal V}(x^*):=\frac{c_{\rm rob}}{c^*}
\end{equation*}
is called the \emph{vulnerability} of design $x^*$ with respect to
random perturbations of the nominal loads.
\end{definition}
\begin{definition}\label{def:rob}
Design $x^*$ (solution of (\ref{eq:ml})) is \emph{robust} with respect
to random perturbations of the nominal loads if its vulnerability is
smaller than or equal to one:
\begin{equation*}
  {\cal V}(x^*)\leq 1 \,.
\end{equation*}

The design is \emph{almost robust} if
\begin{equation*}
  {\cal V}(x^*)\leq 1.05 \,.
\end{equation*}
\end{definition}
The constant $1.05$ gives a 5\% tolerance for non-robustness. Of
course, this constant is to be changed according to particular
applications.

This definition is not only important for the algorithm that follows
but on its own. It gives us the measure of quality (robustness) of a
given design, whether a result of optimization or a manual one, with
respect to random perturbations of the given loads. Furthermore, not
only it will give us an indication whether the design is (almost)
robust---if it is not, we will know by how much. The maximal
``perturbed compliance'' will show by how much our objective value can
get worse under a ``bad'' random perturbation of the given loads.

\subsection{Algorithm for robust design}
The key idea of our approach to finding a robust design is that, for
given $x$ and $\ell$ the eigenvector $f_0^{(\ell)}+P_{\ell}
g^{(\ell)}_{\max}$  \emph{represents the most dangerous load for the
design $x$ and the $\ell$-th load case} in the sense that, under this
load, the compliance is maximized. If the compliance corresponding to
this load is greater than $1.05\cdot c^*$, this load is indeed
dangerous and will be added to our set of load cases; if not, the load
is harmless for the existing design and can be ignored.

This leads to the following algorithm.

\begin{algorithm}{\rm
Finding an almost robust design.
\begin{description}
\item[{\em Step~1.}] Set $s=0$ and ${\cal
    F}=(f^{(1)},\ldots,f^{(L)})$.
\item[{\em Step~2.}] Solve the multiple-load problem (\ref{eq:ml})
  with the original set of loads ${\cal F}$.\\
  Get the optimum design $x_{(0)}$ and compute the associated
  stiffness matrix
  $K(x_{(0)})$.\\
  Compute the norm $\hat{f} = \min_{\ell=1,\ldots,L} \|f^{(\ell)}\|$.\\
  Define the uncertainty ellipsoid by setting $P_\ell$.
\item[{\em Step~3.}] Compute the compliance\\ $c_s =
    \max_{\ell=1,\ldots,L} (f^{(\ell)})^T K(x_{(s)})^{-1}
    f^{(\ell)}$.
\item[{\em Step~4.}] For each load case:
\item[{\em Step~4.1.}] Compute the Schur complement $S^{(\ell)}$
    from (\ref{eq:schur}) and its inverse.
\item[{\em Step~4.2.}] Solve the inhomogeneous eigenvalue problem
    (\ref{eq:ev}) to find the eigenvector $g^{(\ell)}_{\max}$
    associated with the largest eigenvalue.
\item[{\em Step~5.}] Find the index set ${\cal R}$ of all load
    cases with
\begin{equation*}
1.05\cdot c_s < (f_0^{(\ell)}+P_\ell g^{(\ell)}_{\max})^T
K(x_{(s)})^{-1}(f_0^{(\ell)}+P_\ell g^{(\ell)}_{\max})\,.
\end{equation*}
\item[{\em Step~6.}] If ${\cal R}=\emptyset$, then the design is
    almost robust;
  \emph{FINISH}.\\
  If not, add loads with indices $\ell\in{\cal R}$ to the existing
  set of loads
\begin{equation*}
    {\cal F} \leftarrow ({\cal F};g^{(\ell)}_{\max}),\ \ \ell\in{\cal R}.
\end{equation*}
\item[{\em Step~7.}]
  Set $s \leftarrow s+1$.\\
  Solve the problem (\ref{eq:ml}) with loads ${\cal F}$.\\
  Get the optimum design $x_{(s)}$ and compute the associated
  stiffness matrix
  $K(x_{(s)})$.\\
  \emph{Go back to Step 3}.
\end{description}
}
\end{algorithm}

In our numerical experiments, we have solved the inhomogeneous
eigenvalue problems (\ref{eq:ev}) by the power method, as described
below.

\begin{algorithm}{\rm Power method for finding the largest eigenvalue of the
inhomogeneous eigenvalue problem
\begin{equation}\label{eq:ev1}
  Ax - b = \lambda I x
\end{equation}
where $A$ is a real symmetric positive semidefinite $n\times n$ matrix
and $b\in\RR^n$.\\
For $k=1,2,\ldots$ repeat until convergence:
\begin{align}
&y_{k+1} = A x_k - b\\
&\lambda_{k+1} = x_k^Ty_{k+1}\nonumber\\
&x_{k+1} = \frac{y_{k+1}}{\|y_{k+1}\|}\,.\nonumber
\end{align}
}
\end{algorithm}
The convergence proof of the method can be found in
\cite{mattheij-soederlind}. More precisely, the authors show that
$\lambda_k$ converges to the largest eigenvalue $\lambda^*$ and $x_k$
to the associated eigenvector $x^*$, under the condition that the
operator $(I-x^*x^{*T})A(I-x^*x^{*T})/\lambda^*$ is a contraction. In
all our numerical experiments, the power method converged in less than
five iterations, therefore we have not pursued the analysis of this
operator. Furthermore, there is another simple way how to compute all
eigenvalues of (\ref{eq:ev1}), as proposed also by
\cite{mattheij-soederlind}. The problem can be converted into a
quadratic eigenvalue problem which, in turn, can be written as the
following standard (though nonsymmetric) eigenvalue problem:
\begin{equation*}
\begin{bmatrix} 0& I\\ b^Tb - AA^T & 2A\end{bmatrix}
\begin{bmatrix} x\\y\end{bmatrix} = \lambda
\begin{bmatrix} x\\y\end{bmatrix}
\end{equation*}
that can be solved by any standard algorithm. Recall again that the
dimension of these problems is typically very small. Notice that the
above eigenproblem only delivers the eigenvalues of the original
inhomogeneous problem (\ref{eq:ev1}). The associated eigenvectors can
be then computed as $x:=(A-\lambda I)^{-1}b,\ x:=x/\|x\|$.

\section{Numerical examples}\label{sec:ex}
In this section we present numerical examples for robust truss topology
optimization and robust variable thickness sheet problem. Purposely,
all examples are simple enough so that the reader can see the effect of
the robust approach. In fact, for most of our examples the reader will
just guess what the critical perturbations of the nominal loads will
look like. But that is why we have chosen these examples, in order to
show that the results obtained by the algorithm correspond to
engineering intuition. Clearly, for real world problems, the intuition
may not be that obvious.

In all examples, $P_\ell$ was chosen to define a flat uncertainty
ellipsoid around the nominal load:
\[
  {P}_\ell = T^T\begin{bmatrix}1.0\cdot 10^{-3}&0\\0&3.0\end{bmatrix}T\quad
  \mbox{for~} f_0^{(\ell)}= (a,\ b)^T
\]
with
\[
  T=\begin{bmatrix}\cos\phi&\sin\phi\\-\sin\phi&\cos\phi\end{bmatrix},\quad\phi=\arctan(b/a)
\]
see Fig.~\ref{fig:unc}-right.

All optimization problems were solved by our MATLAB based software
package PENLAB\footnote{Downloadable from
http://www.nag.co.uk/projects/penlab} \citep{penlab}.

\subsection{Truss topology optimization}
We first consider the standard ground-structure truss topology
optimization problem. For a given set of potential bars (the ground
structure), we want to find those that best support a given set of
loads. The design variables $x_i$ represent the volumes of the bars
\cite[see e.g.][]{bendsoe}. In our examples, all nodes can be connected
by a potential bar.

\begin{example}\label{ex:scaling}\upshape%
We start with a toy single-load truss topology example shown in
Figure~\ref{fig:1a}-left, together with the ground structure, the
boundary conditions and the nominal load. The obvious solution of the
minimum compliance problem is presented in Figure~\ref{fig:1a}-right; a
single bar in the horizontal direction which is extremely unstable with
respect to any vertical perturbation of the load and its vulnerability
approaches infinity. Also in Figure~\ref{fig:1a}-right we can see the
``most dangerous'' load, as computed by our algorithm. When we add this
load to the set of loads and solve the corresponding two-load problem,
we obtain an optimal design shown in Figure~\ref{fig:1b}-left. This
design is not yet robust as the vulnerability is ${\cal V}=2.25$, still
way bigger than 1.05. Hence we will add the new dangerous load, also
shown in Figure~\ref{fig:1b}-left, to the set of loads and solve a
three-load problem. The optimal design for this problem is shown in
Figure~\ref{fig:1b}-right. This time, the design is robust. For each
iteration of the algorithm, Table~\ref{tab:1} presents: the
corresponding vulnerability ${\cal V}$; maximal compliance for the
current multiple-load problem ``compl''; compliance of the current
design with respect to the nominal load ``compl$_0$''; and the
worst-case load for the previous design, starting with the nominal load
$[10.0,\ 0.0]$.
%-------------------------------------------
\begin{figure}[h]
\begin{center}
 \resizebox{0.33\hsize}{!}
   {\includegraphics{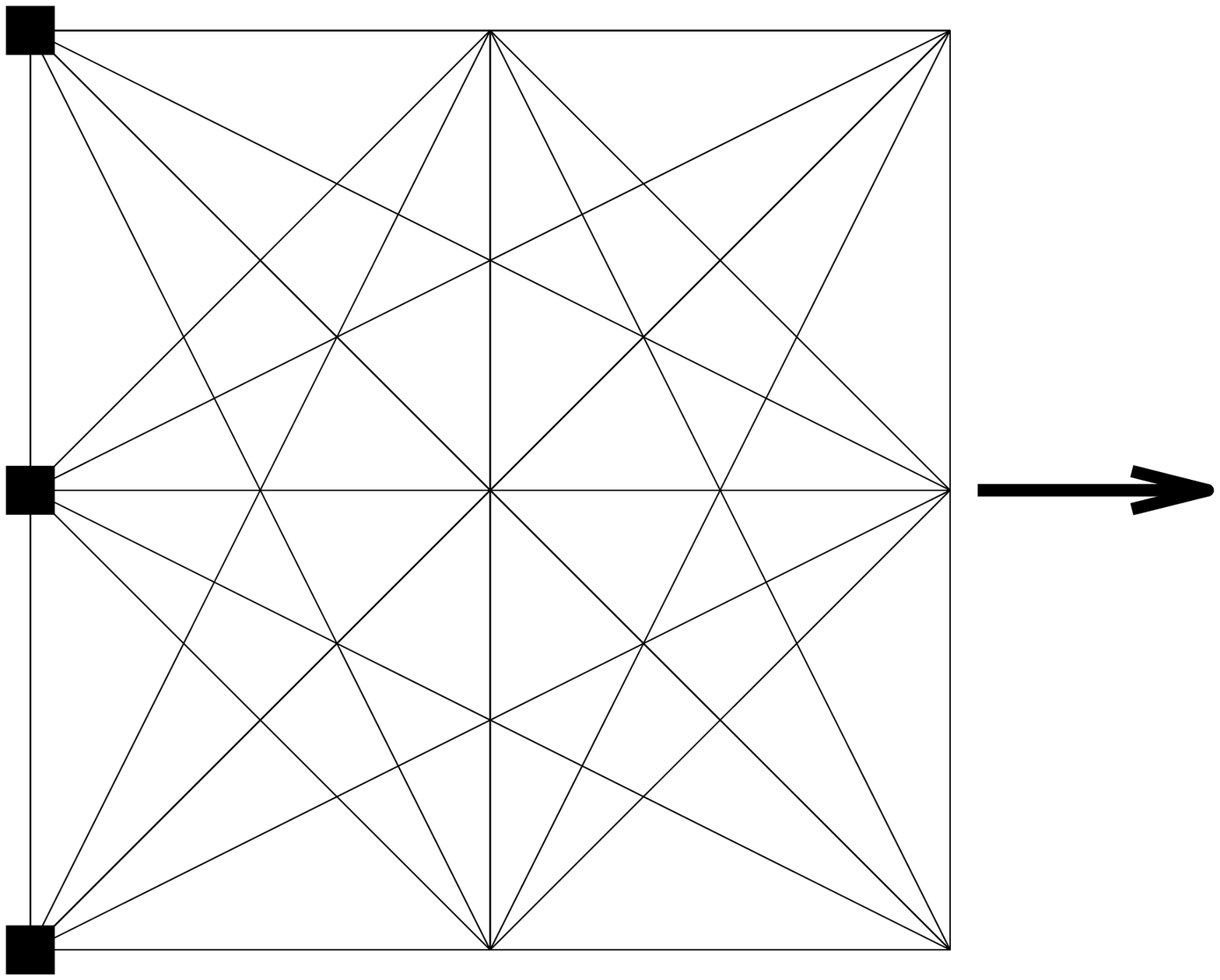}}\qquad
 \resizebox{0.33\hsize}{!}
   {\includegraphics{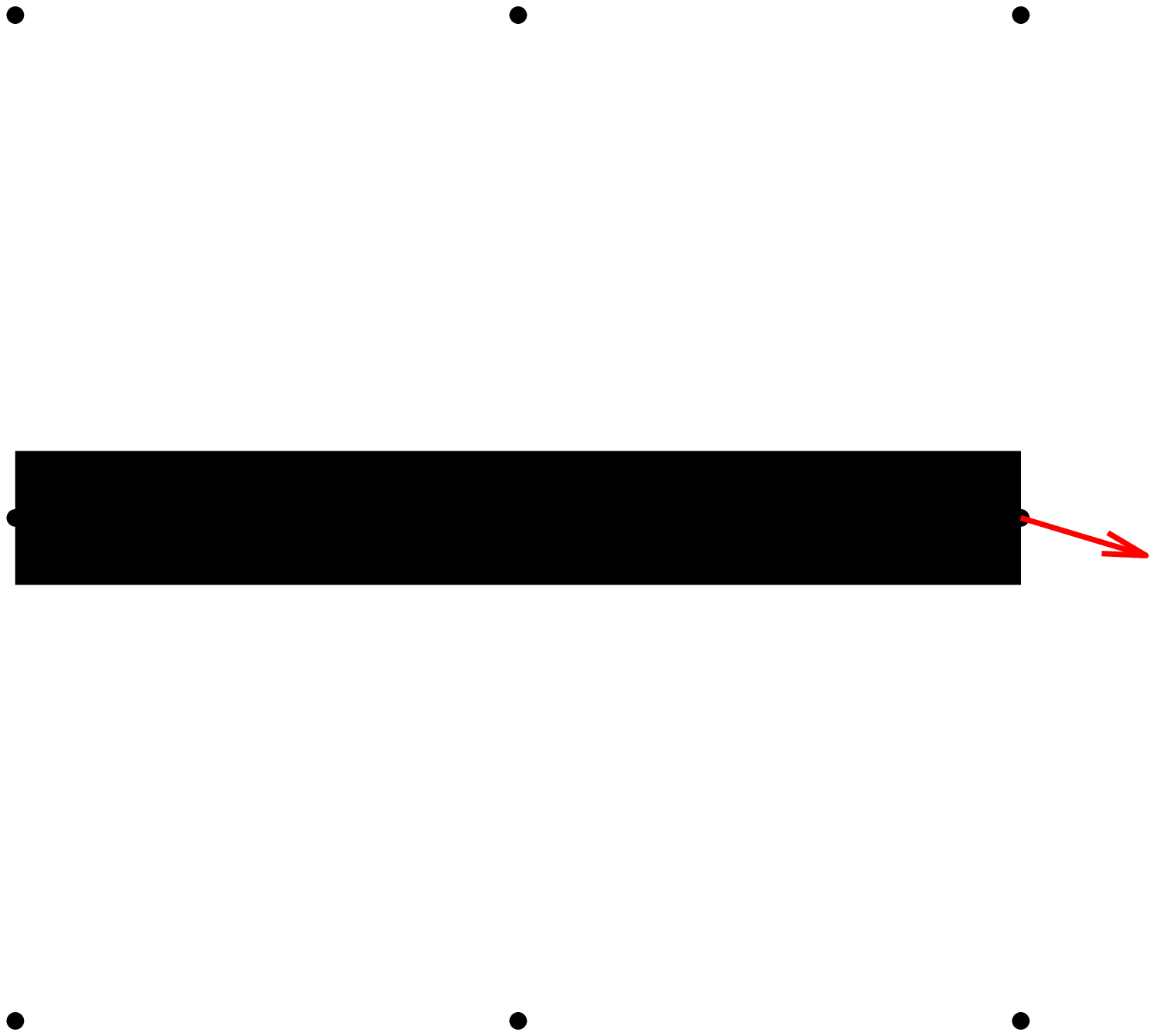}}
  \end{center}
  \caption{Example 1: ground structure, loads and boundary conditions (left)
  and the optimal design, together with the most dangerous perturbation (right).}\label{fig:1a}
\end{figure}
%-------------------------------------------
\begin{figure}[h]
\begin{center}
 \resizebox{0.35\hsize}{!}
   {\includegraphics{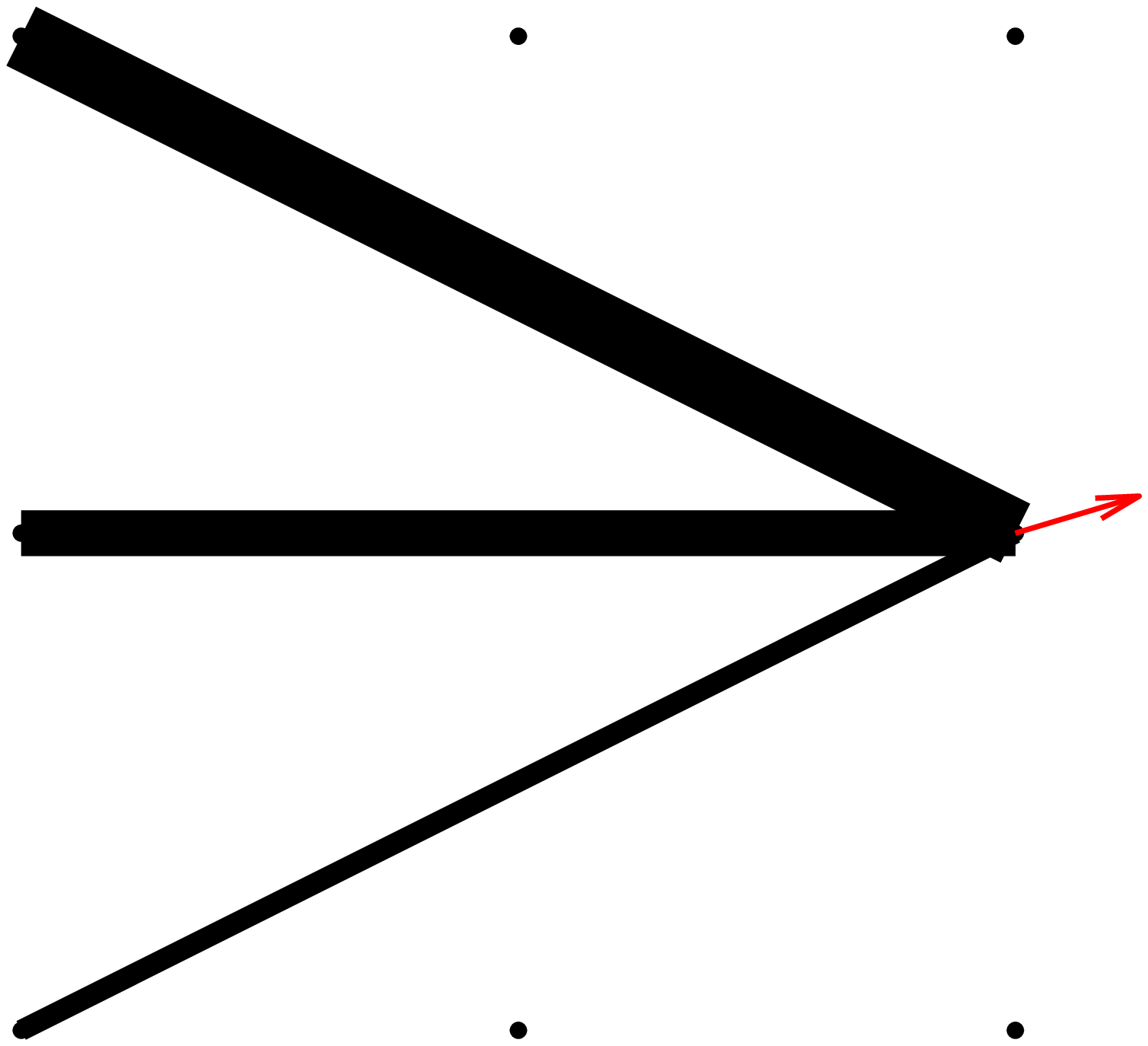}}\qquad
 \resizebox{0.33\hsize}{!}
   {\includegraphics{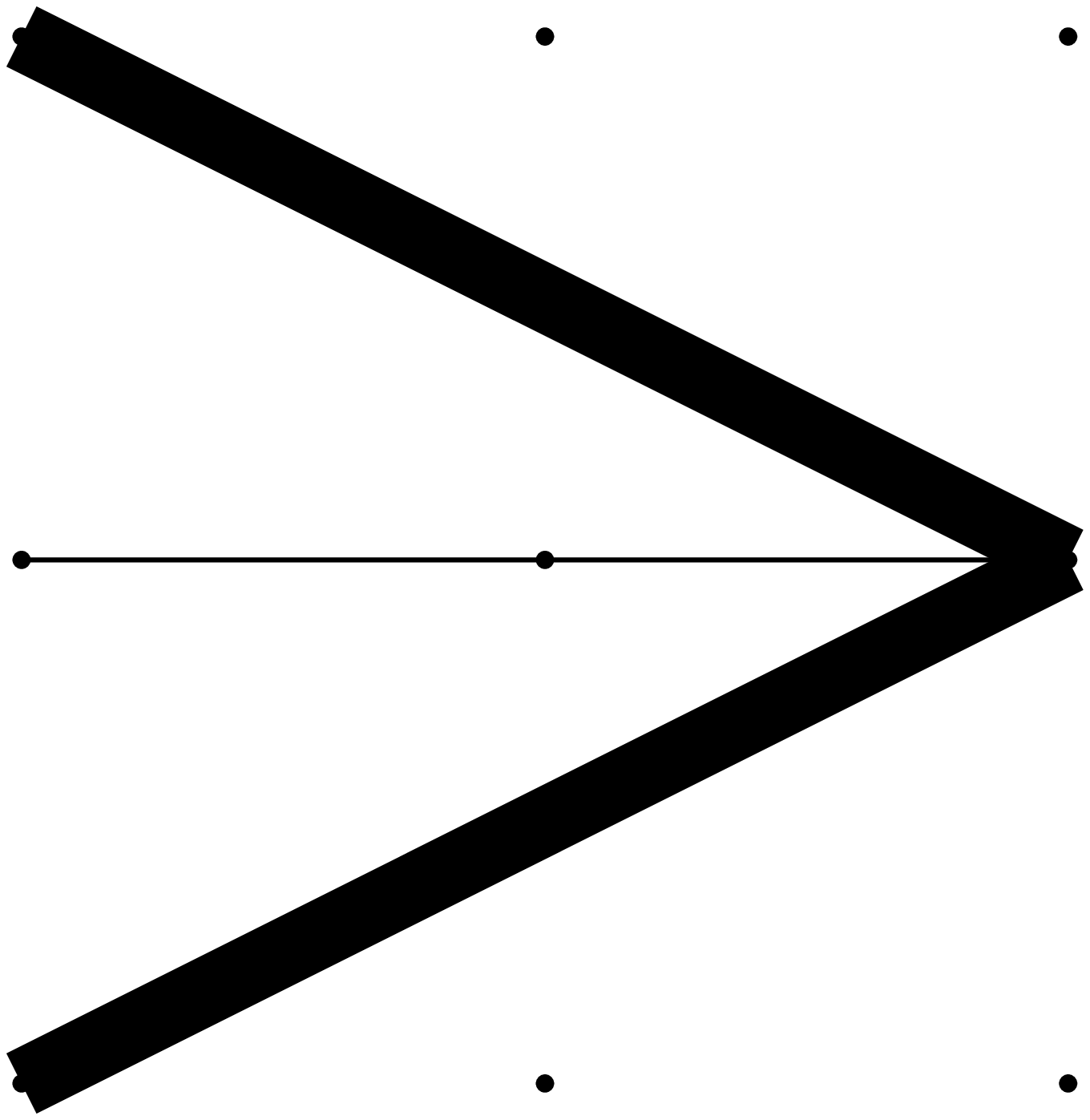}}
  \end{center}
  \caption{Example 1: Optimal design after the first iteration (left)
  and robust optimal design (right).}\label{fig:1b}
\end{figure}
%-------------------------------------------
%-------------------------------------------
\begin{table}[h]
  \centering
  \caption{Example 1: iteration count ``iter'', vulnerability ${\cal V}$,
  maximal compliance of the current problem ``compl'', compliance of
  the current design with respect to the nominal loads ``compl$_0$, and
  the worst perturbation for the previous design $f_s$.}
{\renewcommand{\arraystretch}{1.2}%
  \begin{tabular}{ccccc}
    \hline
    iter &  ${\cal V}$ &  compl & compl$_0$ & $f_s$ \\ \hline
        0&  Inf & 1.0    & 1.0       &  [10.0, 0.0] \\
        1&  2.25 & 1.46   & 1.46      &  [10.0,  3.0] \\
        2&  1.00 & 1.90   & 1.38      &  [10.0, -3.0] \\  \hline
  \end{tabular}\label{tab:1}
}%
\end{table}
%-------------------------------------------
The computed critical perturbation may seem obvious, simply the extreme
perturbation of the nominal force ``up'' and ``down''. Again, that is
why we have chosen this example, in order to show that the results
obtained by the algorithm correspond the engineering intuition.
\end{example}

\begin{example}\label{ex:scaling}\upshape%
We now consider a higher dimensional example of a long slender truss
with 55 nodes and 1485 potential bars. This is again a single-load
problem with a single horizontal force applied at the middle right-hand
side node. The optimal results of the nominal problem and of the robust
problem are shown in Fig.~\ref{fig:2a} left and right, respectively.
%-------------------------------------------
\begin{figure}[h]
\begin{center}
 \resizebox{0.45\hsize}{!}
   {\includegraphics{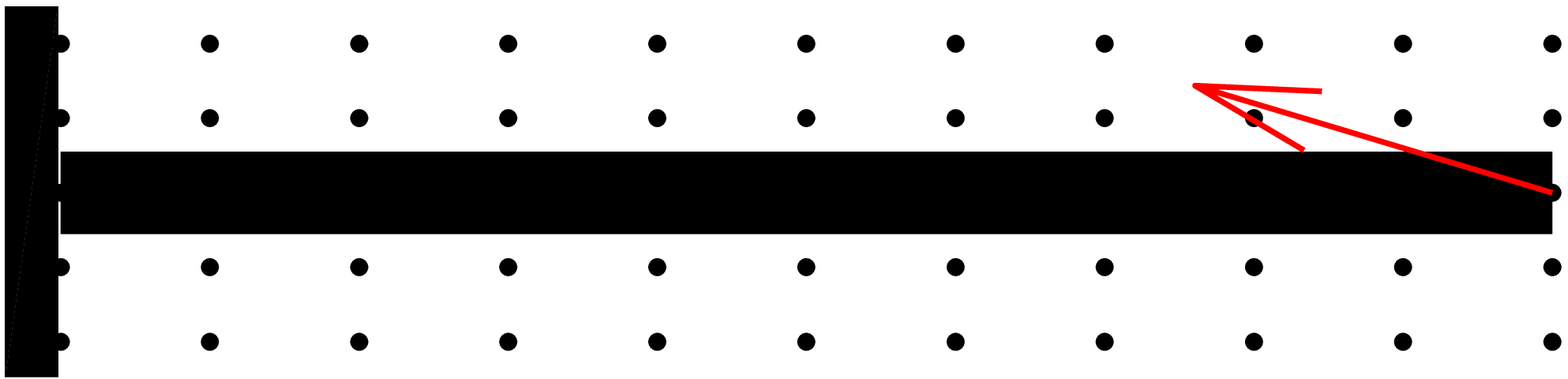}}\qquad
 \resizebox{0.45\hsize}{!}
   {\includegraphics{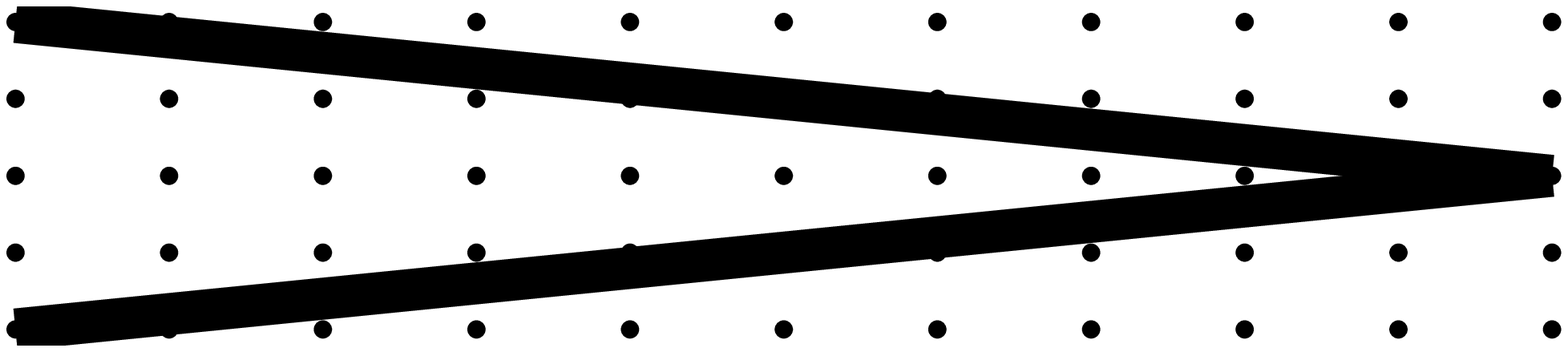}}
  \end{center}
  \caption{Example 2: optimal design for the nominal load,
  together with the most dangerous perturbation (left)
  and robust optimal design (right).}\label{fig:2a}
\end{figure}
%-------------------------------------------
The following Table~\ref{tab:2} shows that we only needed two
iterations of Algorithm~1 to obtain a robust solution.
%-------------------------------------------
\begin{table}[h]
  \centering
  \caption{Example 2, same description as in Table 1}
{\renewcommand{\arraystretch}{1.2}%
  \begin{tabular}{ccccc}
    \hline
    iter &  ${\cal V}$ &  compl & compl$_0$ & $f_s$ \\ \hline
        0&  Inf & 10.0    & 10.0       &  [10.0, 0.0] \\
        1&  3.86 & 90.17   & 64.68      &  [10.0,  3.0] \\
        2&  1.00 & 101.50   & 10.15      &  [10.0, -3.0] \\  \hline
  \end{tabular}\label{tab:2}
}%
\end{table}

\end{example}

\begin{example}\label{ex:scaling}\upshape%
Let us now solve a problem with three load cases, each on them
represented by a single force, as shown in Fig.~\ref{fig:3a}-left. The
ground structure consists of 25 nodes and 300 potential bars.
Fig.~\ref{fig:3a}-right shows the optimal structure for the nominal
loads, as well as the most dangerous perturbations of the nominal loads
for this structure. Due to the ``free'' bar in the top part, this
structure is extremely unstable with respect to perturbations and its
vulnerability tends to infinity, as shown in Table~\ref{tab:3}. After
the first iteration of Algorithm~1, we obtain the truss shown in
Fig.~\ref{fig:3b}-left. This truss is still not robust enough with
respect to the depicted load perturbations and its vulnerability is
${\cal V}=1.55$. Finally, after the second iteration of Algorithm~1, we
obtain the optimal structure shown in Fig.~\ref{fig:3b}-right. This
truss is robust with respect to allowed perturbations.
%-------------------------------------------
\begin{figure}[h]
\begin{center}
 \resizebox{0.36\hsize}{!}
   {\includegraphics{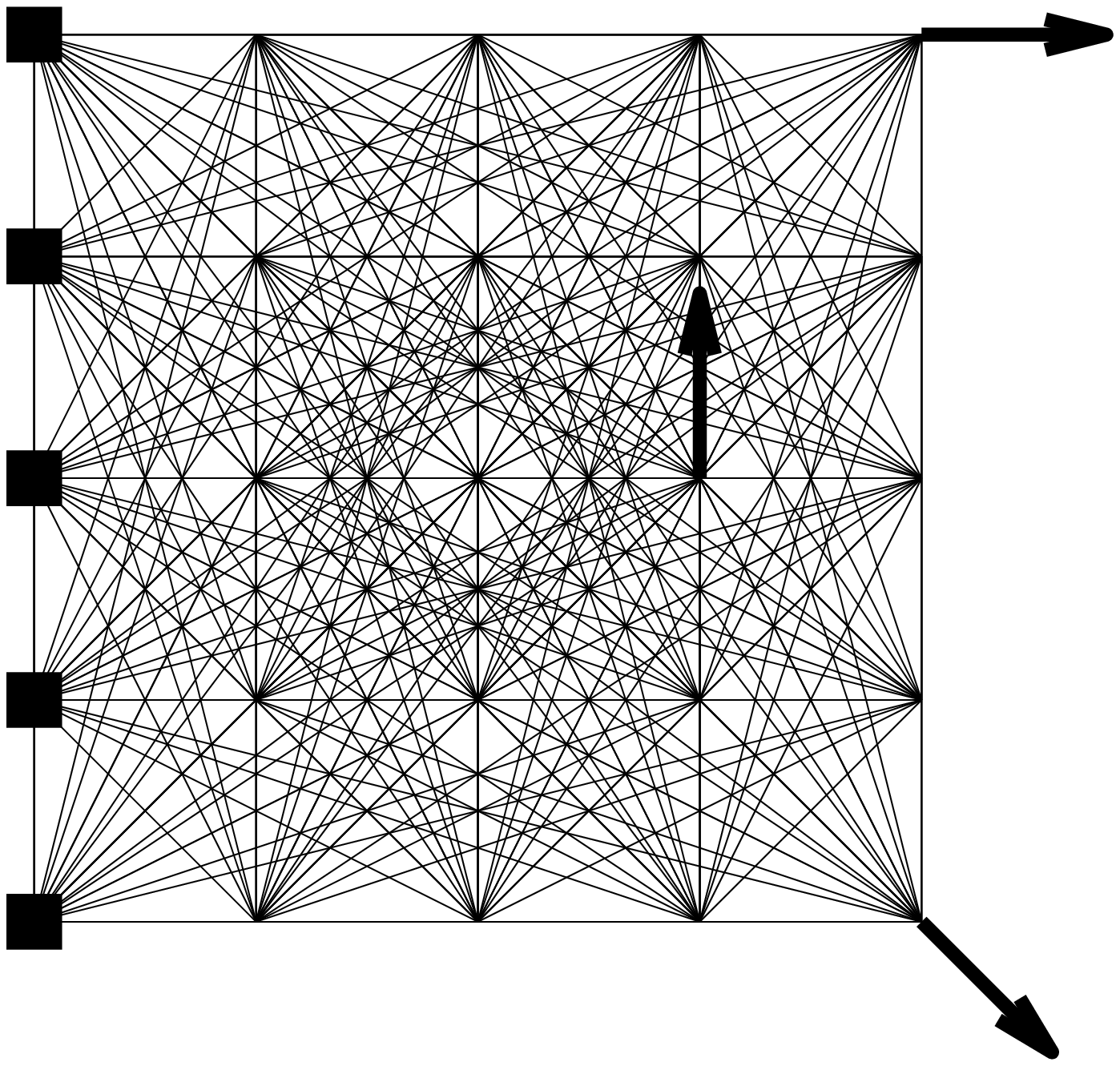}}\qquad
 \resizebox{0.33\hsize}{!}
   {\includegraphics{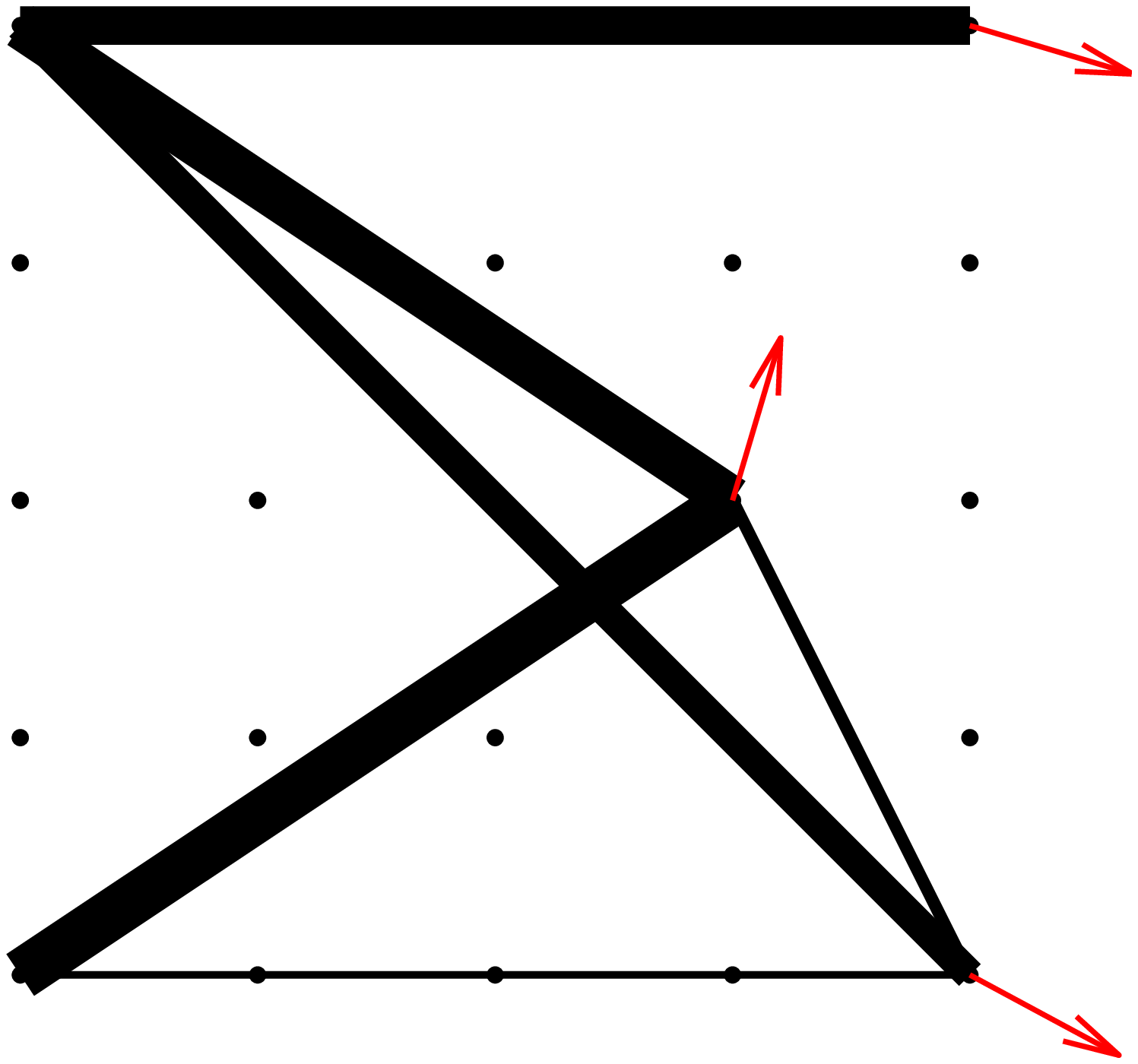}}
  \end{center}
  \caption{Example 3: ground structure, loads and boundary conditions (left)
  and the optimal design, together with the most dangerous perturbation (right).}\label{fig:3a}
\end{figure}
%-------------------------------------------
\begin{figure}[h]
\begin{center}
 \resizebox{0.33\hsize}{!}
   {\includegraphics{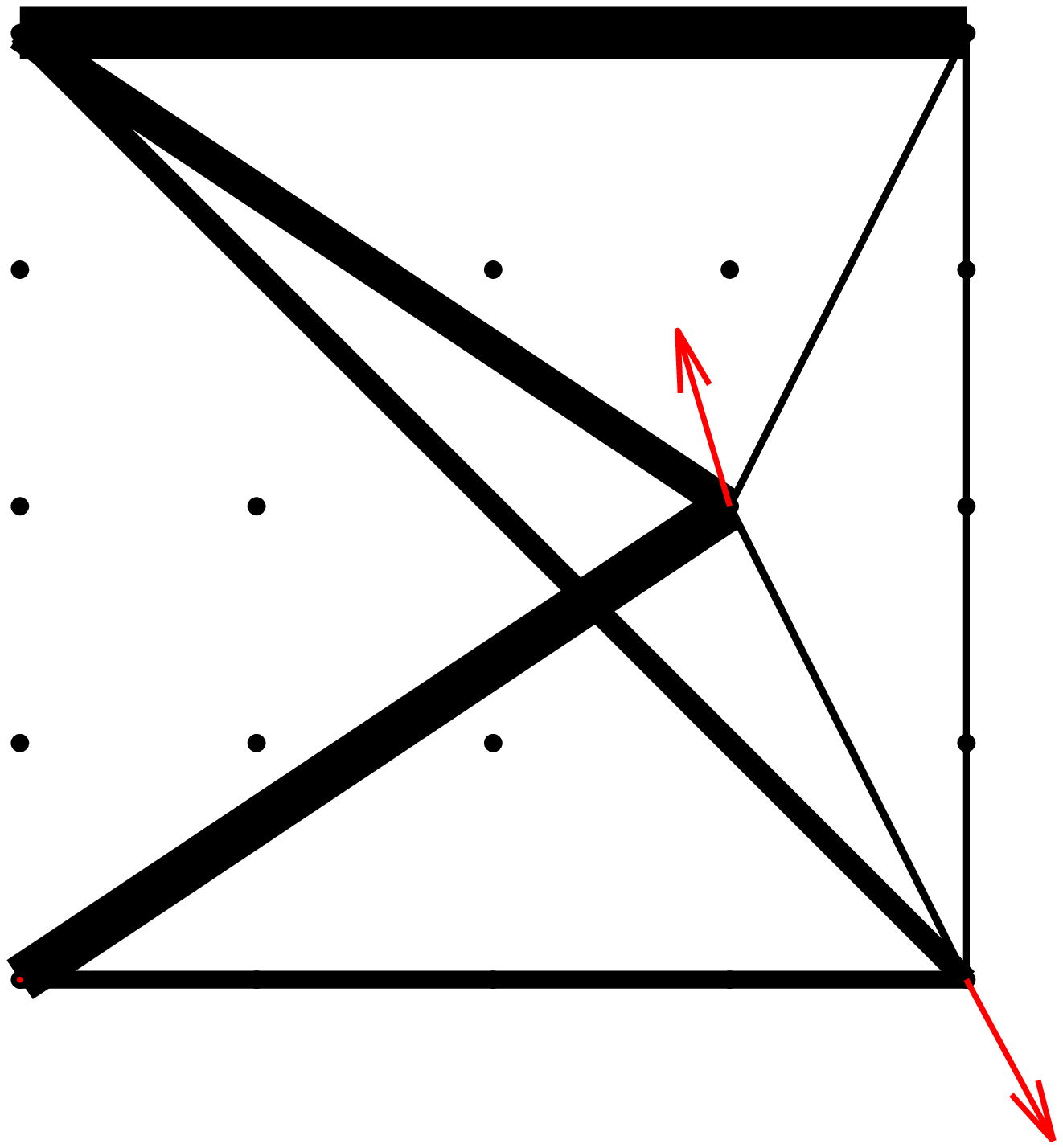}}\qquad
 \resizebox{0.33\hsize}{!}
   {\includegraphics{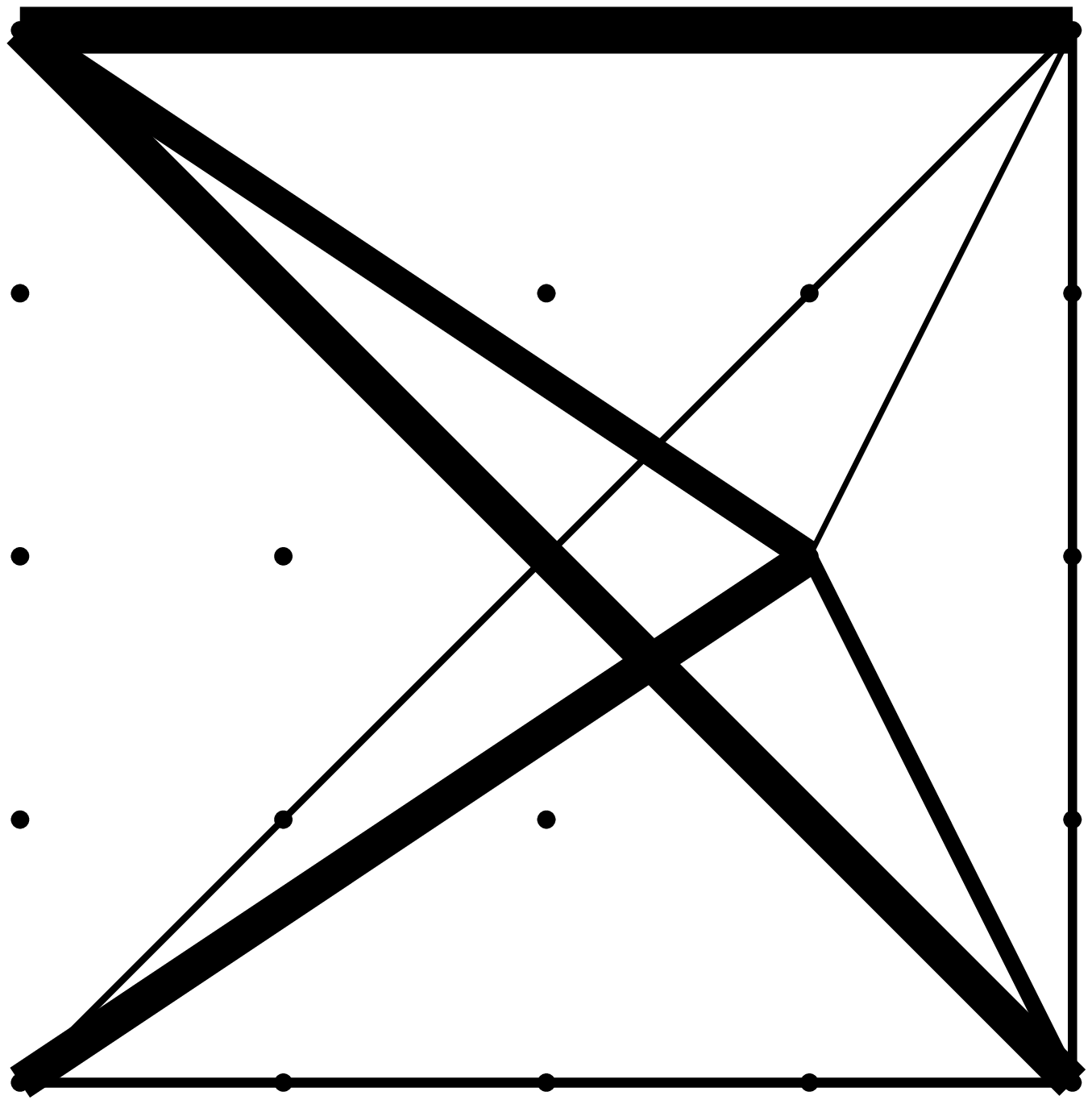}}
  \end{center}
  \caption{Example 3: Optimal design after the first iteration (left)
  and robust optimal design (right).}\label{fig:3b}
\end{figure}
%-------------------------------------------
%-------------------------------------------
\begin{table}[h]
  \centering
  \caption{Example 3, same description as in Table 1}
{\renewcommand{\arraystretch}{1.2}%
  \begin{tabular}{ccccc}
    \hline
    iter &  ${\cal V}$ &  compl & compl$_0$ & $f_s$ \\ \hline
        0&  Inf & 4.82    & 4.82       &  [10, 0]; [0, 10]; [7, -7] \\
        1&  1.55 & 6.08   & 6.08      &  [10,  -2.97]; [2.97, 10]; [9.1, -4.9]\\
        2&  1.00 & 6.61   & 6.30      &  N/A; [-2.97, 10]; [4.9, -9.1]\\  \hline
  \end{tabular}\label{tab:3}
}%
\end{table}

\end{example}

%\begin{example}\label{ex:scaling}\upshape%
%lkj
%%-------------------------------------------
%\begin{figure}[h]
%\begin{center}
% \resizebox{0.36\hsize}{!}
%   {\includegraphics{fig/rob_ex4_ini.eps}}\qquad
% \resizebox{0.33\hsize}{!}
%   {\includegraphics{fig/rob_ex4_1.eps}}
%  \end{center}
%  \caption{...}\label{fig:4a}
%\end{figure}
%
%%-------------------------------------------
%\begin{table}[h]
%  \centering
%  \caption{Example t5x5ff-single}
%{\renewcommand{\arraystretch}{1.2}%
%  \begin{tabular}{ccccc}
%    \hline
%    iter &  rob &  compl & compl$_0$ & $f_s$ \\ \hline
%        0&  1.26 & 7.19  & 7.19      &  [0, -10] \\
%        1&  1.0 & 8.89   & 7.31      &  [3, -10]\\   \hline
%  \end{tabular}\label{tab:4}
%}%
%\end{table}
%
%\end{example}

\subsection{Variable thickness sheet}
In the variable thickness sheet (or free sizing) problem, we consider
plane strain linear elasticity model discretized by the standard finite
element method. The design variables $x_i$ are the thicknesses of the
plate, which are assumed to be constant on each finite element; so we
have as many variables as elements. Again, the model can be found,
e.g.\ in \cite{bendsoe}.

To make the results more transparent, we consider a material with zero
Poisson ratio.

\begin{example}\label{ex:scaling}\upshape%
Consider a rectangular plate as depicted in Fig.~\ref{fig:4a}-left. The
plate is fixed on its left-hand side (by prescribed homogeneous
boundary conditions at the corresponding nodes) and subject to a
horizontal load applied to a small segment in the middle of the
right-hand side edge. Fig.~\ref{fig:4a}-right shows the optimal result
of this single load problem---a single horizontal bar (recall that the
result is due to the zero Poisson ratio). The first line in
Table~\ref{tab:4} shows that this design is far from being robust; its
vulnerability is almost 36. In the same table, in the second row, we
can see the critical perturbation of the three prescribed forces. If we
add these forces as a load number two and solve the corresponding
two-load problem, we obtain an optimal solution depicted in
Fig.~\ref{fig:4b}-left. This solution is still not robust; its
vulnerability is ${\cal V}=3.35$. But after another iteration of
Algorithm~1, we obtain a robust design shown in
Fig.~\ref{fig:4b}-right.
%-------------------------------------------
\begin{figure}[h]
\begin{center}
\begin{tikzpicture}
  [scale=1.5,auto=left]
  \draw[fill=black] (-.15,-.1) rectangle (0,1.1);
  \draw[thick] (0,0) -- (2,0) ;\draw[thick] (2,0) -- (2,1) ;
  \draw[thick] (2,1) -- (0,1) ;\draw[thick] (0,1) -- (0,0) ;
  \draw[thick,->, >=stealth ] (2,0.4) -- (2.4,0.4) ;
  \draw[thick,->, >=stealth ] (2,0.5) -- (2.4,0.5) ;
  \draw[thick,->, >=stealth ] (2,0.6) -- (2.4,0.6) ;
\end{tikzpicture}\qquad
 \resizebox{0.45\hsize}{!}
   {\includegraphics{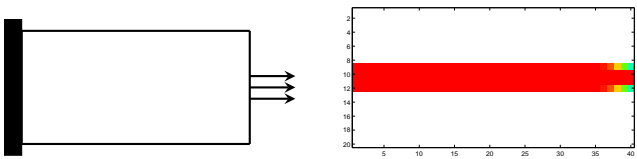}}
  \end{center}
  \caption{Example 4: computational domain, loads and boundary conditions (left)
  and the optimal design, together with the most dangerous perturbation (right).}\label{fig:4a}
\end{figure}
\begin{figure}[h]
\begin{center}
 \resizebox{0.45\hsize}{!}
   {\includegraphics{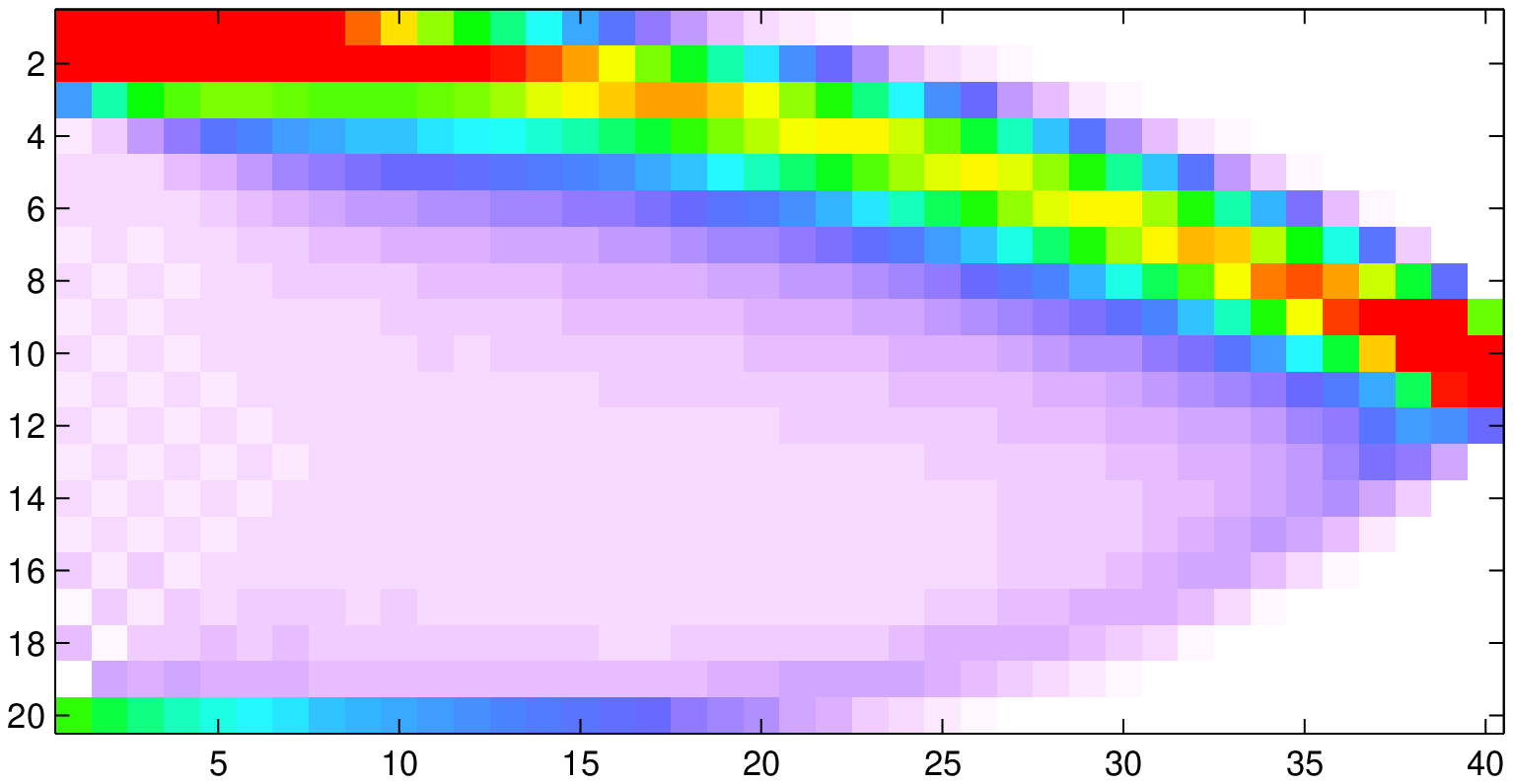}}\qquad
 \resizebox{0.45\hsize}{!}
   {\includegraphics{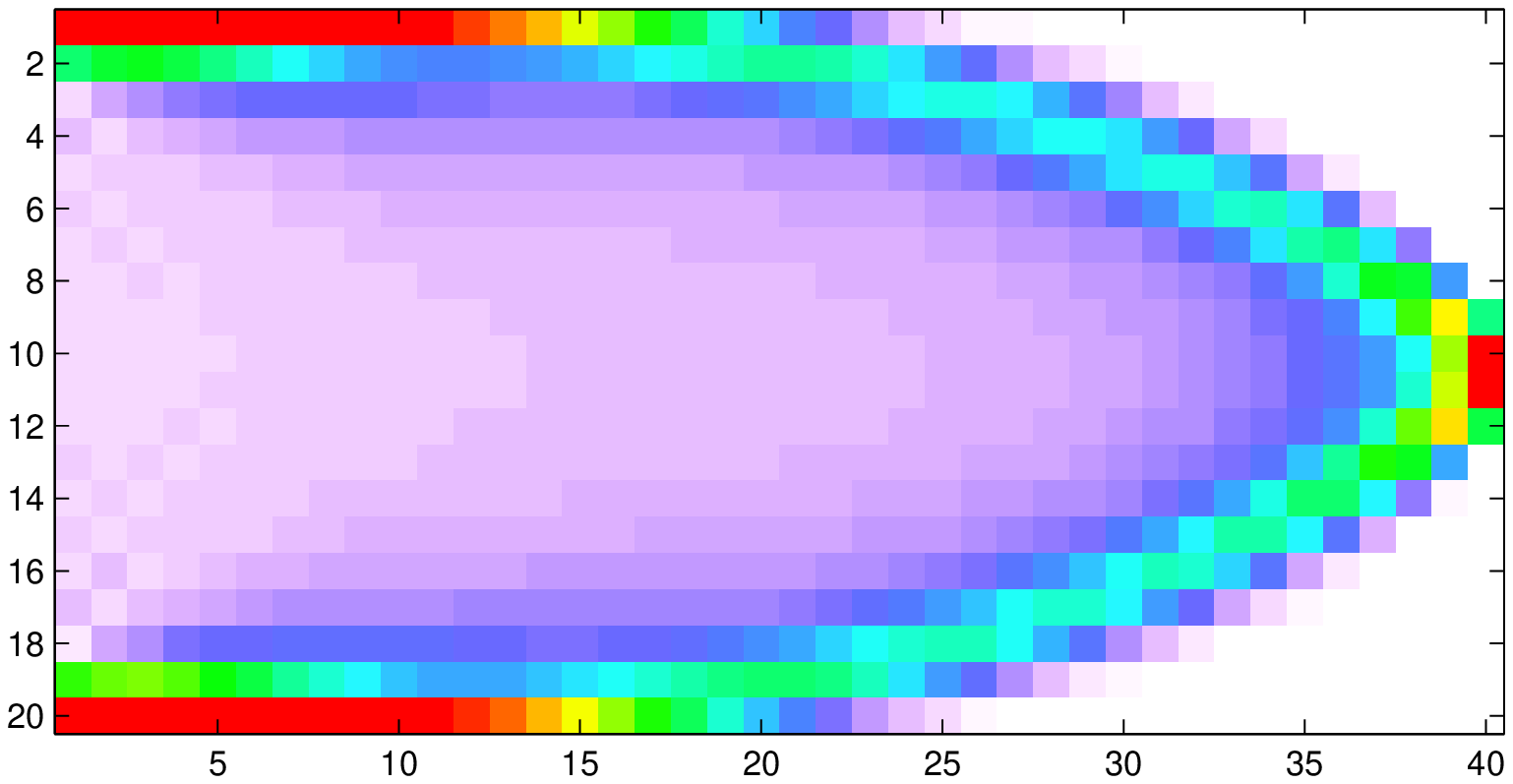}}
  \end{center}
  \caption{Example 4: Optimal design after the first iteration (left)
  and robust optimal design (right).}\label{fig:4b}
\end{figure}

%-------------------------------------------
\begin{table}[h]
  \centering
  \caption{Example 4, same description as in Table 1}
{\renewcommand{\arraystretch}{1.2}%
  \begin{tabular}{ccccc}
    \hline
    iter &  ${\cal V}$ &  compl & compl$_0$ & $f_s$ \\ \hline
        0&  35.93 & 48.88  & 48.88      &  [1, 0, 2, 0, 1, 0] \\
        1&  3.35  & 78.28   & 78.28      &  [1, 0.25, 2, 0.41, 1, 0.56]\\
        2&  1.04  & 111.80   & 56.54      &  [1, -0.42, 2, -0.42, 1, -0.43]\\   \hline
  \end{tabular}\label{tab:4}
}%
\end{table}

\end{example}

%
%\begin{acknowledgement}
%The authors would like to thank two anonymous referees for their comments helping to improve the presentation.
%This work ...
%\end{acknowledgement}

\bibliography{truss_vib_smo}

\end{document}